\documentclass[12pt,reqno]{amsart}

\usepackage[colorlinks]{hyperref}
\usepackage{amssymb}
\usepackage{mathrsfs}
\usepackage[all,cmtip]{xy}
\usepackage{cite}
\usepackage{bm}
\usepackage{enumerate}

\newtheorem{thm}{Theorem}[section]
\newtheorem{cor}[thm]{Corollary}
\newtheorem{lem}[thm]{Lemma}
\newtheorem{prop}[thm]{Proposition}

\newtheorem{rem}[thm]{Remark}

\theoremstyle{definition}
\numberwithin{equation}{section}

\newcommand{\ts}{\mathbb{T}}

\newcommand{\si}{ \sigma}
\newcommand{\Si}{ \Sigma}
\newcommand{\ti}{ \mathbb{T}^\infty}
\newcommand{\ghat}{ \widehat{G}}
\newcommand{\lif}{ L^\infty}
\newcommand{\sn}{ \sigma_N}
\newcommand{\nsn}{ \left|\sn\right|}
\newcommand{\bigls}{ \left(}
\newcommand{\bigrs}{ \right)}
\newcommand{\lifr}{ \lif_{\mathbb{R}}}
\newcommand{\tsn}{T_{\sn}}

\newcommand{\T}{\mathbb{T}}

\allowdisplaybreaks

\title[ ]{The first Szeg\H{o} limit theorem on multi-dimensional torus}

\author{Kunyu Guo}
\address{Kunyu Guo: School of Mathematical Sciences, Fudan University, Shanghai, 200433, China}
\email{kyguo@fudan.edu.cn}

\author{Dilong Li}
\address{Dilong Li: School of Mathematical Sciences, Fudan University, Shanghai, 200433, China}
\email{dlli23@m.fudan.edu.cn}

\author{Qi Zhou}
\address{
Qi Zhou: School of Mathematical Sciences, Soochow University, Suzhou, 215006, China}
\email{zhouqi@suda.edu.cn}


\makeatletter
\@namedef{subjclassname@2020}{\textup{2020} Mathematics Subject Classification}
\makeatother

\begin{document}

\begin{abstract}
In this paper, we consider the first Szeg\H{o} limit theorems on $d$-torus $\mathbb{T}^d$ for $1\leq d\leq +\infty$. It is shown that for any F{\o}lner sequence $\{\sn\}$ of $\mathbb{Z}^d$ and $\varphi\in L^1_+(\mathbb{T}^d)$, it holds that
$$
\lim_{N\rightarrow \infty}\left(\det T_{\sn}\varphi\right)^{\frac{1}{\nsn}}=\exp\left(\int_{\mathbb{T}^d} \log\varphi~dm_{d}\right).
$$
In the case  $d=+\infty$, we are associated with multiplicative Toeplitz matrix $\bm{T \varphi}=\{\bm{\widehat{\varphi}}(j/i)\}_{i,j\in\mathbb{N}}$ and the most concerned non-F{\o}lner truncation, that is, $\bm{T_N \varphi}=\{\bm{\widehat{\varphi}}(j/i)\}_{1\leq i,j\leq N}$, where $\sn=\{1,\dots,N\}$.
It is shown that for each $\varphi\in \lifr(\mathbb{T^{\infty}})$ and $f\in C[\text{ess-inf} ~\varphi,~\text{ess-sup}~\varphi]$, the limit $\lim_{N\rightarrow \infty} \frac{1}{N}\mathrm{Tr} f \big(\bm{T_N \varphi}\big)$ exsits.
Moreover, it is proven that the limit $\lim_{N\rightarrow \infty}\left(\det\bm{T_N \varphi}\right)^{\frac{1}{N}}$ exists for any $\varphi\in L^1_+(\mathbb{T}^\infty)$ with strictly positive essential infimum. These results are directly related to two problems posed by Nikolski and Pushnitski in \cite{2021SZEGO}.
\end{abstract}

\maketitle

\section{Introduction}\label{section_one}
Let $\varphi\in L^1(\mathbb{T})$ with Fourier coefficients
$$
\widehat{\varphi}(k)=\int_{\ts} \varphi(z)\bar{z}^k~dm_1,~~k\in\mathbb{Z},
$$
where $m_1$ is the normalized Lebesgue measure on the unit circle $\mathbb{T}$. The Toeplitz matrix with symbol $\varphi$ is defined as
$$
T\varphi=\{\widehat{\varphi}(j-i)\}_{i,j=0,1,2,~\ldots}.
$$
Let $T_N \varphi$ be the truncated $N \times N$ Toeplitz matrices, i.e.
$$
T_N \varphi=\left\{\widehat{\varphi}(j-i)\right\}_{i,j=0}^{N-1}.
$$
Based on the asymptotics of the determinants of $T_N \varphi $, the following theorem was first proven by Szeg\H{o}  in 1915 \cite{szego1915ein}.
\begin{thm}\label{first_version}[First Szeg\H{o} Limit Theorem, First Version].
    Let $\varphi(z)>0$ be a continuous function on $\mathbb{T}$, then
    \begin{equation}\label{det_form}
        \lim_{N\rightarrow +\infty}\left(\det T_N\varphi\right)^{\frac{1}{N}}=\exp\left(\int_{\mathbb{T}}\log\varphi~dm_1\right).
    \end{equation}
\end{thm}
\noindent Actually, Theorem \ref{first_version} can be extended to non-negative integrable symbols, i.e. $\varphi\in L^1_+(\mathbb{T})$ (see \cite{simon2005orthogonal}). It is worth mentioning that the estimation of determinants of Toeplitz matrices is closely related to an extraordinary variety of problems in mathematics, physics, and engineering. There is a vast literature on these subjects (see  \cite{bottcher1983invertibility,bottcher2013analysis,bottcher2012introduction,fisher1969toeplitz,widom1973toeplitz,widom1976asymptotic,widom1975limit}).

Note that \eqref{det_form} can be rewritten in the form
\begin{equation}\label{log_form}
    \lim_{N\rightarrow +\infty} \frac{1}{N} \text{Tr} \log (T_N \varphi )=\int_{\mathbb{T}} \log \varphi~dm_1.
\end{equation}
This leads to a second version of the limit theorem \cite{szegHo1920beitrage}.
\begin{thm}\label{second_version}[First Szeg\H{o} Limit Theorem, Second Version].
    Let $\varphi(z)\in L^\infty_{\mathbb{R}}(\mathbb{T})$. Then for any $f\in C[\mathrm{ess\text{-}inf}~\varphi,\mathrm{ess\text{-}sup}~\varphi]$, it holds that
    \begin{equation}\label{f_form}
        \lim_{N\rightarrow +\infty} \frac{1}{N} \mathrm{Tr} f (T_N\varphi) = \int_{\mathbb{T}} f(\varphi)~dm_1.
    \end{equation}
\end{thm}

\noindent In the literature, there is a third version of the first Szeg\H{o} limit theorem (see \cite[Theorem 2.7.13]{simon2005orthogonal}).
\begin{thm}\label{third_version}[First Szeg\H{o} Limit Theorem, Third Version].
    Let $\varphi \in L^{1}_+(\mathbb{T})$. If $f$ is continuous on $[0,+\infty)$ and $\lim\limits_{x\rightarrow +\infty} \frac{f(x)}{x}$ is finite,
    then
    \begin{equation}\label{orig_sze}
        \lim_{N\rightarrow +\infty} \frac{1}{N} \mathrm{Tr} f (T_N\varphi) = \int_{\mathbb{T}} f(\varphi)~dm_1.
    \end{equation}
\end{thm}
\noindent
In general, the first Szeg\H{o} limit theorem characterizes asymptotic spectral distributions for Toeplitz matrices defined by symbols on the unit circle $\mathbb{T}$. More precisely, Let $\{\lambda_{j,N}\}_{j=1}^N$ denote the eigenvalues of $T_N\varphi$ counted with multiplicities. Then the first version shows that the geometric mean of the eigenvalues $\{\lambda_{j, N}\}_{j=1}^N$ converges to $\exp(\int \log\varphi ~dm_1)$, that is,
$$
\lim_{N\rightarrow+\infty}\left(\prod_{j=1}^N\lambda_{j,N}\right)^{\frac{1}{N}}=\exp\left(\int_{\mathbb{T}}\log\varphi~dm_1\right).
$$
The other two versions indicate that
$$
\frac{1}{N}\sum_{j=1}^N\delta_{\lambda_{j,N}}\rightarrow \varphi_* m_1,~~~~\text{in the weak* topology},
$$
where $\varphi_* m_1$ is the push-forward of $m_1$, that is, $\varphi_*m_1(\Delta)=m_1(\varphi^{-1}(\Delta))$ for any Borel set $\Delta\subset[\mathrm{ess\text{-}inf}~\varphi,\mathrm{ess\text{-}sup}~\varphi]$.
Roughly speaking, these three versions are not equivalent to each other, for which each version has distinct assumptions for symbols.

In this paper, motivated by \cite{2021SZEGO,nikolski2021szego}, we will generalize the first Szeg\H{o} limit theorem on the unit circle $\mathbb{T}$ to Toeplitz matrices defined by symbols on $d$-torus $\mathbb{T}^d$, $1\leq d\leq +\infty$. In particular, the case $d=+\infty$ exactly corresponds to multiplicative Toeplitz matrices that have constantly received extensive attention (see \cite{hilberdink2017,nikolski2021szego,hilberdink2023}).

The $d$-torus $\mathbb{T}^d$ is a compact Abelian group in multiplication. Let us first recall some results on Szeg\H{o} theorems associated with compact Abelian groups.

Let $G$ be a compact Abelian group and let $m$ be its normalized Haar measure. The dual group of $G$ is denoted by $\ghat$. For any $\varphi\in L^1(G)$, the Laurent matrix with symbol $\varphi$ is defined as
$$
L\varphi=\left\{\widehat{\varphi}(\xi -\gamma)\right\}_{ \gamma, \xi \in \ghat }.
$$
It is known that $L\varphi$ is bounded on $L^2(G)$ if and only if $\varphi \in L^\infty(G)$. In this case, $L\varphi$
is the matrix form of $M(\varphi)$, the operator of multiplication defined by $\varphi$ on $L^2(G)$.
Given any finite subset $\sigma=\{\xi_1\cdots,\xi_N\}\subset\ghat$, the truncated  matrix with respect to $\sigma$ is
$$T_\sigma \varphi =\left\{\widehat{\varphi}(\xi_j -\xi_i)\right\}_{1\leq i, j \leq N}.$$
Here we use the same notation as Toeplitz matrices for simplicity. A sequence of finite subsets $\sn\subset\ghat$ is called a \textit{F{\o}lner sequence} of $\ghat$ if for any $\xi\in\ghat$
\begin{equation*}
    \lim_{N\rightarrow +\infty}\frac{\left|(\xi+\sn)\cap\sn\right|}{\nsn}= 1,
\end{equation*}
where $\left|\sigma\right|$ is the number of elements in $\sigma$. The space of all real-valued functions in $L^p(G)$ is denoted by $L^p_{\mathbb{R}}(G)$.

Building on an idea of Arveson \cite{arveson1994c}, B\'edos proved a Szeg\H{o}-type theorem on compact Abelian groups, which reads as follows.
\begin{thm}\label{szego_thm_group}\cite[Theorem 11]{bedos1997folner}
    Let $\varphi\in\lifr(G)$ and suppose that $\{\sn\}$ is a F{\o}lner sequence of $\ghat$. Then for any $f\in C[\mathrm{ess\text{-}inf}~\varphi,\mathrm{ess\text{-}sup}~\varphi]$, it holds that
    \begin{equation}\label{g_Folner_limit}
        \lim_{N\rightarrow +\infty}\frac{1}{\nsn}\mathrm{Tr} f\big( \tsn \varphi\big)=\int_G f(\varphi) ~dm.
    \end{equation}
\end{thm}

\noindent Applying Theorem \ref{szego_thm_group} to $\mathbb{T}^d$ for $1\leq d\leq +\infty$ gives the following theorem.
\begin{thm}\cite[Theorem 3.1]{2021SZEGO}\label{second_T_d}
        Let $\varphi\in L^\infty_{\mathbb{R}}(\mathbb{T}^d)$ and suppose that $\{\sn\}$ is a F{\o}lner sequence of $\mathbb{Z}^d$. Then for any $f\in C[\mathrm{ess\text{-}inf}~\varphi,\mathrm{ess\text{-}sup}~\varphi]$, one has that
    \begin{equation}\label{second_T_d_limit}
        \lim_{N\rightarrow +\infty}\frac{1}{\nsn}\mathrm{Tr} f\big( \tsn \varphi\big)=\int_{\mathbb{T}^d} f(\varphi) ~dm_d,
    \end{equation}
    where $m_d$ is the normalized Haar measure on $\mathbb{T}^d$.
\end{thm}
\noindent
This is exactly the counterpart of Theorem \ref{second_version} in the case  $\T^d$. Naturally,  problems are left to find the analogs for the first and the third versions of the first Szeg\H{o} limit theorems on $d$-torus $\mathbb{T}^d$, $1\leq d\leq +\infty$, which is the focus of this paper.

We first establish the following inequality.
\begin{thm}\label{bdd_below_thm}
    Let $\varphi\in L^1_+(\mathbb{T}^d)$, $1\leq d\leq+\infty$, and let $\sigma$ be a finite subset of $\mathbb{Z}^d$, then
    \begin{equation}\label{bdd_bdd_below_equ}
        \exp\left(\int_{\mathbb{T}^d}\log\varphi~dm_d\right)\leq\left(\det T_\sigma \varphi\right)^{\frac{1}{|\sigma|}}\leq \|\varphi\|_1.
    \end{equation}
\end{thm}
\noindent
By Theorem \ref{bdd_below_thm} and a cluster-point argument, we will prove the following theorem, which corresponds to Theorem \ref{first_version} with symbol $\varphi\in L^1_+(\mathbb{T})$.
\begin{thm}\label{first_T_d}
    Let $\varphi\in L^1_+(\mathbb{T}^d)$, $1\leq d\leq+\infty$, and let $\{\sn\}$ be a F{\o}lner sequence of $\mathbb{Z}^d$. Then
    \begin{equation}\label{first_T_d_limit}
        \lim_{N\rightarrow +\infty}\left(\det T_{\sn}\varphi\right)^{\frac{1}{\nsn}}=\exp\left(\int_{\mathbb{T}^d}\log\varphi~dm_d\right).
    \end{equation}
\end{thm}
\noindent When $d=+\infty$, Theorem \ref{first_T_d} positively answers one problem posed by Nikolski and Pushnitski in \cite{2021SZEGO}. In particular, if $\varphi=|f|$ for some $f\in H^1(\mathbb{T}^d)$, then $\log\varphi\in L^1(\mathbb{T}^d)$ (for the case  $\ti$, see \cite[Corollary 2]{aleman2019fatou}), and hence the limit is non-zero.

\begin{thm}\label{third_T_d}
     Let $\varphi\in L^1_+(\mathbb{T}^d)$, $1\leq d\leq+\infty$, and let $\{\sn\}$ be a F{\o}lner sequence of $\mathbb{Z}^d$.  If $f\in C[0,\infty)$ such that $\lim\limits_{x\rightarrow +\infty} \frac{f(x)}{x}$ is finite,
    then
    \begin{equation}\label{third_T_d_limit}
    \lim_{N\rightarrow +\infty}\frac{1}{\nsn}\mathrm{Tr} f\bigls\tsn\varphi\bigrs=\int_{\mathbb{T}^d}f(\varphi)~dm_d.
    \end{equation}
\end{thm}
\noindent Theorem \ref{third_T_d} is exactly the third version of Szeg\H{o}'s theorem on $d$-torus $\mathbb{T}^d$ for $1\leq d\leq+\infty$, which corresponds to Theorem \ref{third_version}.

\vskip 2mm
As well known, the case $\ti$ is the most attractive in recent years (see \cite{seip2022,pushnitski2018,kosz2023,yan2022}). Recall that $\widehat{\ti}=\mathbb{Z}^\infty$. Here
 $$
 \mathbb{Z}^{\infty}=\{\kappa=(\kappa(1), \kappa(2), \ldots): \kappa(j)\in \mathbb{Z}, \text{ and only finitely many } \kappa(j) \text{ are non-zero}\}.
 $$
The set $\mathbb{Z}_{+}^{\infty}$ is defined to be
 $$
 \mathbb{Z}_{+}^{\infty}=\{\kappa\in\mathbb{Z}^\infty:\text{ all the }\kappa(j)\geq 0\}.
 $$
 As done for the classical Toeplitz matrices on the unit circle $\mathbb{T}$, in the case $\ti$,
  the Toeplitz matrices involved are labeled by $\mathbb{Z}_+^\infty$, i.e. $T\varphi=\{\widehat{\varphi}(\kappa'-\kappa)\}_{\kappa,\kappa'\in \mathbb{Z}^\infty_+}$. Recall that a finite or infinite matrix $M=\{m_{i,j}\}$ is a \textit{multiplicative Toeplitz matrix} if $m_{i,j}=m_{l,k}$ whenever $i/j=l/k$. In what follows we will rewrite $T\varphi$ as a multiplicative Toeplitz matrix. To be more precise, let $p_1, p_2, p_3, \dots$ be the ordered sequence of all prime numbers. The fundamental theorem of arithmetic shows that  every  $q\in \mathbb{Q}_+$ can be uniquely expressed as
$$q=p_1^{\alpha_1}p_2^{\alpha_2}\cdots=:\bm{P}^{\alpha(q)},$$ where $\alpha(q)=(\alpha_1, \alpha_2, , \ldots)\in \mathbb{Z}^{\infty}.$
It is easy to verify that
\begin{equation}\label{homomorphism}
    \alpha: (\mathbb{Q}_+, \times) \rightarrow (\mathbb{Z}^{\infty},+),~~~q \mapsto \alpha(q)
\end{equation}
is a group isomorphism.
For $\varphi \in L^1(\mathbb{T}^{\infty})$, via the above isomorphism, we can relabel the Fourier coefficients by $\mathbb{Q}_+$ and write it in boldface font as
$$\bm{\widehat{\varphi}}(q)=\widehat{\varphi}(\alpha(q)),
~~~q\in \mathbb{Q}_+.$$
This turns the Toeplitz matrix $T\varphi$ to the following multiplicative form
$$
\bm{T}(\bm{\varphi})=\{\bm{\widehat{\varphi}}(j/i)\}_{i,j \in \mathbb{N}},
$$
i.e. a multiplicative Toeplitz matrix.
The multiplicative Toeplitz matrices can be dated back to Toeplitz \cite{toeplitz1938theorie} and have been revitalized recently due to their close connection with Dirichlet series. We refer the readers to \cite{hedenmalm1995hilbert,bayart2002,seip2019} for a more in-depth discussion on this topic.

In the case of multiplication, a sequence $\{\sigma_N\}_{N=1}^{\infty}$ of finite sets of $\mathbb{N}$ is a \textit{multiplicative F{\o}lner sequence},
if for any $k\in \mathbb{N}$
$$
\frac{\left|(k\sigma_N) \cap \sigma_N\right|}{\left|\sigma_N\right|} \rightarrow 1,~~~\text{as}~N\rightarrow +\infty.
$$
In the case $d=+\infty$,  Theorem \ref{first_T_d} and Theorem \ref{third_T_d} can be revisited in a multiplicative form.
\begin{thm}\label{first_multi}
    Let $\varphi\in L^1_+(\mathbb{T}^\infty)$ and let $\{\sn\}$ be a multiplicative F{\o}lner sequence of $\mathbb{N}$. Then
    $$
    \lim_{N\rightarrow +\infty}\left(\det \bm{T_{\sn}\varphi}\right)^{\frac{1}{\nsn}}=\exp\left(\int_{\mathbb{T}^\infty}\log\varphi~dm_\infty\right).
    $$
\end{thm}
\begin{thm}\label{third_multi}
    Let $\varphi\in L^1_+(\mathbb{T}^\infty)$ and let $\{\sn\}$ be a multiplicative F{\o}lner sequence of $\mathbb{N}$. If $f\in C[0,\infty)$ such that $\lim\limits_{x\rightarrow +\infty} \frac{f(x)}{x}$ is finite,
    then
    $$
    \lim_{N\rightarrow +\infty}\frac{1}{\nsn}\mathrm{Tr} f\bigls\bm{\tsn\varphi}\bigrs=\int_{\mathbb{T}^\infty}f(\varphi)~dm_\infty.
    $$
\end{thm}

\noindent
Our preceding results, Theorem \ref{first_multi} and Theorem \ref{third_multi}, heavily depend on F{\o}lner condition. In what follows  we consider the most concerned non-F{\o}lner sequence $\sn=\{1,\dots,N\},~N=1,2,\dots$, and the corresponding truncation
$$
\bm{T_N}\bm{\varphi}=\left\{\bm{\widehat{\varphi}}(j/i)\right\}_{i,j=1}^{N}.
$$
\begin{thm}\label{second_non_folner}
 Let $\varphi\in \lifr(\mathbb{T^{\infty}})$. Then there exists a probability measure $\mu$ on $[\mathrm{ess\text{-}inf}~\varphi,\mathrm{ess\text{-}sup}~\varphi]$ such that
for any $f\in C[\mathrm{ess\text{-}inf}~\varphi,\mathrm{ess\text{-}sup}~\varphi]$,
$$
\lim_{N\rightarrow +\infty} \frac{1}{N}\mathrm{Tr} f \big(\bm{T_N \varphi}\big)=\int f~d\mu.
$$
In particular, if $\varphi$ depends only on finitely many variables, the corresponding measure $\mu$ concentrates on a countable subset.
\end{thm}
\noindent
This can be regarded as the second version of Szeg\H{o}'s theorem for this non-F{\o}lner case, and affirmatively answers another problem posed by Nikolski and Pushnitski in \cite{2021SZEGO}.

As well known, there exists a deep connection between the determinant of $ \bm{T_N}\bm{\varphi}$ and the determinants of number-theoretic matrices (see \cite{hilberdink2006determinants,balazard2019fonctions}). As an application of Theorem \ref{second_non_folner}, we will prove the following theorem, which can be seen as the first version of Szeg\H{o}'s theorem in the case of $\T^\infty$.
\begin{thm}\label{first_non_folner}
If $\varphi\in L^1_+(\mathbb{T}^\infty)$ with $\mathrm{ess\text{-}inf}~\varphi>0$, then
     $$    \lim_{N\rightarrow +\infty}\left(\det\bm{T_N \varphi}\right)^{\frac{1}{N}}~\text{exists, and is finite}.
    $$
\end{thm}

This paper is organized as follows. Section \ref{sec_two} gives the proof of Theorem \ref{bdd_below_thm} and contains some preliminary results, which are essential for the proof of Theorem \ref{third_T_d} and Theorem \ref{first_non_folner}. Theorem \ref{first_T_d} is established in Section \ref{sec_three}. We consider the most concerned non-F{\o}lner cases in Section \ref{sec_four}, and prove Theorem \ref{second_non_folner}
 and Theorem \ref{first_non_folner} in this section. We also apply the above results to  Gram determinants of dilation systems in Section  \ref{sec_five}.

\section{Proof of Theorem \ref{bdd_below_thm} and some preliminary results}\label{sec_two}

We begin with the proof of Theorem \ref{bdd_below_thm}.

\begin{proof}[Proof of Theorem \ref{bdd_below_thm}]
    Let $\{\lambda_j\}_{j=1}^{|\si|}$ be eigenvalues of $T_\si\varphi$ by counting multiplicities. Then
    $$
        \left(\det T_\sigma \varphi\right)^{\frac{1}{|\sigma|}}=\left(\prod_{j=1}^{|\si|}\lambda_j\right)^{\frac{1}{|\si|}}\leq \frac{1}{|\si|}\sum_{j=1}^{|\si|}\lambda_j
        =\frac{1}{|\si|}\mathrm{Tr}T_\si\varphi=\|\varphi\|_1.
    $$
    It remains to prove the bounded-below part, that is,
    \begin{equation}\label{bdd_below_equ} \exp\left(\int_{\mathbb{T}^d}\log\varphi~dm_d\right)\leq \left(\det T_\sigma \varphi\right)^{\frac{1}{|\sigma|}}.
    \end{equation}

   \noindent Without loss of generality, assume that $d=+\infty$. We first prove that \eqref{bdd_below_equ} holds for each $\varphi\in L^\infty_+(\mathbb{T}^\infty)$.
    Let $\mathbb{Z}^0=\{(0,0,\dots)\}\subset \mathbb{Z}^\infty$ and
     $$
     \mathbb{Z}^n=\{\kappa=(\kappa(1),\dots,\kappa(n),0,0, \dots):~\kappa(j)\in\mathbb{Z},~1\leq j\leq n\}\subset\mathbb{Z}^\infty.
     $$
     Define $n(\si)=\min\{n:\si\subset\mathbb{Z}^n\}$. Since $\si$ is a finite subset, $n(\sigma)$ is well-defined. We will give the proof by induction on $n(\si)$.

     When $n(\si)=0$, $\si=\{(0,0,\dots)\}$ and \eqref{bdd_below_equ} follows from Jensen's inequality.
     Assume that \eqref{bdd_below_equ} holds for any $\varphi\in L^\infty_+(\mathbb{T}^\infty)$ and $\si$ such that $n(\si)\leq m-1$.
     Given $\varphi\in L^\infty_+(\mathbb{T}^\infty)$ and $\si$ such that $n(\si)=m$, write $\si_j=\{\kappa\in\si:\kappa(m)=j\}$. Then,
     $$
     \si=\bigsqcup_{j=-\infty}^{+\infty}\si_j.
     $$
     We may further assume that $\log\varphi\in L^1(\mathbb{T}^\infty)$ and define
     $$
     \varphi_r(\bm{\theta})=\exp\left(\frac{1}{2}\int_{0}^{2\pi}\log\varphi(\bm{\theta}_t)\frac{1+r e^{i(\theta_m-t)}}{1-r e^{i(\theta_m-t)}}~\frac{dt}{2\pi}\right),~~~~0<r<1,
     $$
     where $\bm{\theta}_t=(\theta_1,\dots,\theta_{m-1},t,\theta_{m+1},\dots)$.
     Therefore,
     $$
     \left|\varphi_r(\bm{\theta})\right|^2=\exp\left(\int_{0}^{2\pi}\log\varphi(\bm{\theta}_t)\frac{1-r^2}{1-2r \cos (\theta_m-t)+r^2}~\frac{dt}{2\pi}\right).
     $$
     By Fatou's theorem (see \cite[pp. 34]{hoffman}), $|\varphi_r|^2\rightarrow \varphi$ a.e.  on $\ti$ as $r\rightarrow 1$.
     Note that for any $r$, $\|\varphi_r^2\|_\infty\leq\|\varphi\|_\infty$, and hence $|\varphi_r|^2\rightarrow \varphi$ in $L^1(\mathbb{T}^\infty)$ as $r\rightarrow 1$.
     Then
     \begin{equation}\label{det_var_r}
         \left(\det T_\si\varphi\right)^{\frac{1}{|\si|}}=\lim_{r\rightarrow 1}\left(\det T_\si\left|\varphi_r\right|^2\right)^{\frac{1}{|\si|}}.
     \end{equation}
     Let $E_\si=\mathrm{span}\{\bm{z}^\kappa=z_1^{\kappa(1)}z_2^{\kappa(2)}\cdots:\kappa\in\si\}\subset L^2(\mathbb{T}^\infty)$ and let $P_\si$ be the orthogonal projection onto $E_\si$. We have
     $$
     T_\si \left|\varphi_r\right|^2=P_\si M(\varphi_r)^*
     M(\varphi_r)|_{E_\si}\geq P_\si M(\varphi_r)^*P_\si P_\si M(\varphi_r)|_{E_\si}=\left(T_\si\varphi_r\right)^* \left(T_\si\varphi_r\right).
     $$
     This yields that
     \begin{equation}\label{det_inequ}
         \left(\det T_\si\left|\varphi_r\right|^2\right)^{\frac{1}{|\si|}}\geq \left|\det T_\si\varphi_r\right|^{\frac{2}{|\si|}}.
     \end{equation}
    For each $\kappa\in \mathbb{Z}^\infty$ with  $\kappa(m)<0$, we derive   $\widehat{\varphi_r}(\kappa)=0$ from the definition of  $\varphi_r$,  and hence we can write $T_\si \varphi$ as a block upper triangular matrix by a unitary isomorphism,
     \begin{equation}
         T_\si\varphi_r\cong
         \left(\begin{matrix}
         \ddots&&&\\
               &T_{\si_j}\varphi_r&\ast&\cdots\\
               &&T_{\si_{j+1}}\varphi_r&\cdots\\
            0  &&&\ddots
         \end{matrix}\right).
     \end{equation}
    Therefore it holds true
     \begin{equation}\label{det_prod}
         \det T_\si\varphi_r=\prod_{\si_j\neq \emptyset}\det T_{\si_j}\varphi_r.
     \end{equation}
     Define
     $$
     \widetilde{\si_j}=\{\kappa=(\kappa(1),\dots,\kappa(m-1),0,\dots): \exists ~\kappa'\in\si_j, ~\text{s.t.} ~\kappa(j)=\kappa'(j),1\leq j\leq m-1\}.
     $$
     Setting
     $$
     \widetilde{\varphi}(\bm{\theta})=\exp\left(\frac{1}{2}\int_{0}^{2\pi}\log\varphi(\bm{\theta}_t)~\frac{dt}{2\pi}\right),
     $$
     then $\widetilde{\varphi}\in L^\infty_+(\mathbb{T}^\infty)$, and $T_{\si_j}\varphi_r=T_{\widetilde{\si_j}}\widetilde{\varphi}$ since $\widehat{\varphi_r}(\kappa)=\widehat{\widetilde{\varphi}}(\kappa)$ for any $\kappa$ with $\kappa(m)=0$.
     Therefore, we have
     \begin{equation}\label{det_prod_tilde}
         \det T_\si\varphi_r=\prod_{\si_j\neq \emptyset}\det T_{\widetilde{\si_j}}\widetilde{\varphi}.
     \end{equation}
     Since for each $\widetilde{\si_j}$, $n(\widetilde{\si_j})\leq m-1$. We apply the previous assumption to each pair $(\widetilde{\si_j},\widetilde{\varphi})$. Then,
     \begin{align*}
         \det T_\si\varphi_r&\geq \exp\left(\left(\sum_j|\widetilde{\si_j}|\right)\int_{\mathbb{T}^\infty}\log\widetilde{\varphi}~dm_\infty\right)\\
         &= \exp\left(\frac{|\si|}{2}\int_{\mathbb{T}^\infty}\log\varphi~dm_\infty\right).
     \end{align*}
     Combined with \eqref{det_var_r} and \eqref{det_inequ}, \eqref{bdd_below_equ} follows. Thus, \eqref{bdd_below_equ} holds for any finite subset $\si\subset\mathbb{N}$ and $\varphi\in L^\infty_+(\mathbb{T}^\infty)$.

     For $\varphi\in L^1_+(\mathbb{T}^\infty)$, define $\varphi_k=\min\{\varphi,k\}$. Then
     $$
     \exp\left(\int_{\mathbb{T}^\infty}\log\varphi_k~dm_{\infty}\right)\leq\left(\det T_\si\varphi_k\right)^{\frac{1}{|\si|}} .
     $$
     Letting $k\rightarrow+\infty$, we have the theorem.
\end{proof}

The following lemma shows that the existence of the limits in \eqref{second_T_d_limit}  for continuous functions is equivalent to the existence of the limits for polynomials. The proof is standard, which is omitted here.
\begin{lem}\label{general_prop}
Let $\varphi\in\lifr(\mathbb{T}^d)$ and let $\Si=\{\sn\}$ be a sequence of finite subsets of $\mathbb{Z}^d$. Then the following are equivalent.
\begin{enumerate}[(1)]
 \item For any integer $n\geq 0$, $$\lim_{N\rightarrow +\infty} \frac{1}{\nsn} \mathrm{Tr}\bigls\tsn\varphi\bigrs^n\text{~exists}.$$
 \item For any $f\in C[\mathrm{ess\text{-}inf}~\varphi,\mathrm{ess\text{-}sup}~\varphi]$, $$\lim_{N\rightarrow +\infty} \frac{1}{\nsn} \mathrm{Tr} f\bigls\tsn\varphi\bigrs \text{~exists}.$$
\end{enumerate}
\end{lem}
\noindent
In what follows we discuss $L^1_+(\mathbb{T}^d)$ symbols with strictly positive essential infimum.
A sequence of finite subsets $\Si=\{\sn\}$ is called a \textit{hypo-F{\o}lner sequence} if for any $\varphi\in\lifr(\mathbb{T}^d)$, $(\varphi,\Si)$ satisfies either (1) or (2) of Lemma \ref{general_prop}. Every F{\o}lner sequence is a hypo-F{\o}lner sequence. As shown in \cite{nikolski2021szego,2021SZEGO}, there exists a lot of non-F{\o}lner sequences that are hypo-F{\o}lner sequences. If $\Si$ is a hypo-F{\o}lner sequence, by Lemma \ref{general_prop},
for any $\varphi\in L_+^\infty(\mathbb{T}^d)$ with $\text{ess-inf}~\varphi >0$,
the limit
\begin{equation}\label{PF_log_linf}
\Delta_\Sigma(\varphi)=\lim_{N\rightarrow +\infty}\bigls\det\tsn\varphi\bigrs^{\frac{1}{\nsn}}~~~\text{exists}.\end{equation}
We will simply use $\Delta(\varphi)$ when no ambiguity arises. The following theorem is a generalization of the first version of the first Szeg\H{o} limit theorem on multi-dimensional torus $\mathbb{T}^d$, $1\leq d\leq+\infty$.
\begin{thm}\label{convergence_L_one}
    Let $\varphi\in L^1_+(\mathbb{T}^d)$ with $\mathrm{ess\text{-}inf}~\varphi>0$, and let $\Si=\{\sn\}$ be a hypo-F{\o}lner sequence.
    Then
    $$
\lim_{N\rightarrow +\infty}\bigls\det\tsn\varphi\bigrs^{\frac{1}{\nsn}}~~~\text{exists}.
$$
Further, if we define $\varphi_n(x):=\min\{\varphi(x),n\}$ for $1\leq n<+\infty$,
then
\begin{equation}\label{l1_lim}
\lim_{N\rightarrow +\infty}\bigls\det\tsn\varphi\bigrs^{\frac{1}{\nsn}}=\lim_{n\rightarrow +\infty}\Delta(\varphi_n).
\end{equation}
\end{thm}

\begin{proof}
    Without loss of generality, we assume that $\varphi\geq 1$.
    Denote by
    $$1 \leq \lambda_1  \leq \cdots \leq \lambda_{\nsn}$$ the eigenvalues of $\tsn\varphi$ by considering multiplicities, and
    $$1 \leq \lambda_1^{(n)}  \leq \cdots \leq \lambda_{\nsn}^{(n)}$$
    the eigenvalues of $\tsn\varphi_n$ by considering multiplicities. Since $\varphi_n\leq \varphi$, we have $\tsn\varphi_n \leq \tsn\varphi$. By \cite[pp. 319]{lax2002functional}, it follows that $$\lambda_i^{(n)}\leq\lambda_i,~~~i=1,  \ldots , \nsn.$$
    Then
    $$\begin{aligned}
    &\frac{1}{\nsn}\text{Tr}\log\tsn\varphi- \frac{1}{\nsn}\text{Tr}\log\tsn\varphi_n\\
    =&\frac{1}{\nsn}  \sum_{i=1}^{\nsn} \left(\log \lambda_i - \log  \lambda_i^{(n)} \right)\\
    \leq& \frac{1}{\nsn}  \sum_{i=1}^{\nsn}  \left(\lambda_i - \lambda_i^{(n)} \right)\\
    = &\frac{1}{\nsn}\text{Tr} \tsn\varphi- \frac{1}{\nsn}\text{Tr}\tsn\varphi_n\\
    = &\|\varphi-\varphi_n\|_1.
    \end{aligned}$$
    It gives
    \begin{align}
        \begin{aligned}\label{control}
        \frac{1}{\nsn}\text{Tr}\log\tsn\varphi_n&\leq \frac{1}{\nsn}\text{Tr}\log\tsn\varphi\\
        &\leq \|\varphi-\varphi_n\|_1+\frac{1}{\nsn}\text{Tr}\log\tsn\varphi_n.
    \end{aligned}
    \end{align}
    Similarly, for $n\leq m$ we have
    $$
    0\leq \frac{1}{\nsn}\text{Tr}\log\tsn\varphi_m- \frac{1}{\nsn}\text{Tr}\log\tsn\varphi_n\leq \|\varphi_m-\varphi_n\|_1.
    $$
    This implies that $\{\log\Delta(\varphi_n)\}_{n=1}^{\infty}$ is an increasing convergence sequence.
Combined with \eqref{control}, it holds that
$$
\lim_{N\rightarrow\infty}\frac{1}{\nsn}\text{Tr}\log\tsn\varphi=\lim_{n\rightarrow\infty}\log\Delta(\varphi_n),
$$
and hence \eqref{l1_lim} follows.
\end{proof}

\begin{cor}\label{Folner_thm_first}
    Let $\varphi\in L^1_+(\mathbb{T}^d)$ with $\mathrm{ess\text{-}inf}~\varphi>0$, and let $\Si=\{\sn\}$ be a F{\o}lner sequence. Then
    $$
    \lim_{N\rightarrow +\infty}\bigls\det\tsn\varphi\bigrs^{\frac{1}{\nsn}}=\exp\bigls\int_{\mathbb{T}^d}\log\varphi~dm_d\bigrs.
    $$
\end{cor}
\begin{proof}
Let $\varphi_n(x)=\min\{\varphi(x),n\}$ for $1\leq n<+\infty$. Applying Levy's lemma, one has that
$$
\lim_{n\rightarrow +\infty}\Delta(\varphi_n)=\lim_{n\rightarrow +\infty}\exp\bigls\int_{\mathbb{T}^d}\log\varphi_n~dm_d\bigrs
=\exp\bigls\int_{\mathbb{T}^d}\log\varphi~dm_d\bigrs.
$$
Thus, the conclusion follows from Theorem \ref{convergence_L_one}.
\end{proof}

\begin{proof}[Proof of Theorem \ref{third_T_d}]
    By Corollary \ref{Folner_thm_first}, \eqref{third_T_d_limit} holds for   $f^{(c)}(x)=\log(x+c)$, $c>0$. The rest proof of Theorem \ref{third_T_d} is nearly the same as the classical case $d=1$ (see \cite[Theorem 2.7.13]{simon2005orthogonal}), which we omit here.
\end{proof}

\section{Proof of Theorem \ref{first_T_d}}\label{sec_three}

This section is devoted to the proof of Theorem \ref{first_T_d}.  To this end, we need several auxiliary lemmas.
\begin{lem}\label{one_over_f}
Let $1\leq d\leq+\infty$, $\varphi\in L^1_+(\mathbb{T}^d)$ with $\mathrm{ess\text{-}inf}~\varphi=\delta >0$, and let $\Sigma=\{\sigma_N\}$ be a F{\o}lner sequence. Then
\begin{equation}\label{lem_equ}
    \lim_{N\rightarrow +\infty}\frac{1}{\nsn}\mathrm{Tr}\bigls\tsn \varphi\bigrs^{-2}=\int_{\mathbb{T}^d}\frac{1}{\varphi^2}~dm_d.
\end{equation}
\end{lem}

\begin{proof}
Let $\psi=\varphi-\delta$ and $f(x)=(x+\delta)^{-2}$. By Theorem \ref{third_T_d}, we have
\begin{equation}
    \lim_{N\rightarrow +\infty}\frac{1}{\nsn}\mathrm{Tr}f(T_{\sn}\psi)=\int_{\mathbb{T}^d}f(\psi)~dm_d.
\end{equation}
Hence, \eqref{lem_equ} holds.
\end{proof}

The following lemma comes from \cite[Lemma 7.18]{hagen2000c}.
\begin{lem} \label{prod_lmm}
Let $\varphi_1,\dots,\varphi_n\in\lifr(\mathbb{T}^d)$ and    let $\Sigma=\{\sigma_N\}$ be a F{\o}lner sequence.
Then
$$
\lim_{N\rightarrow+\infty}\frac{1}{\nsn}\mathrm{Tr}\left(\prod_{i=1}^n\tsn \varphi_i \right)=\int_{\mathbb{T}^d} \left(\prod_{i=1}^n \varphi_i \right) ~dm_d.
$$
\end{lem}

\begin{lem}\label{g_over_f}
Let $\psi\in\lifr(\mathbb{T}^d)$ and $\varphi\in L^1_+(\mathbb{T}^d)$ with $\mathrm{ess\text{-}inf}~\varphi= \delta>0$. Suppose that $\Si=\{\sn\}$ is a F{\o}lner sequence, then for any $n\geq 0$,
\begin{equation}\label{g_over_f_equ}
    \lim_{N\rightarrow+\infty}\frac{1}{\nsn}\mathrm{Tr} \left(\tsn \psi \cdot(\tsn \varphi)^{-1} \right)^n=\int_{\mathbb{T}^d}\left(\frac{\psi}{\varphi}\right)^n~dm_d.
\end{equation}
\end{lem}
\begin{proof}
It is trivial when $n=0$.
When $n=1$, we define $\varphi_k(x)=\min\{\varphi(x),k\}$.
Then,
\begin{align}\begin{aligned}\label{lem_n_first}
        &\left|\frac{1}{\nsn}\text{Tr} \left(\tsn \psi\cdot(\tsn \varphi)^{-1}\right)-\frac{1}{\nsn}\text{Tr}\left(\tsn \psi\cdot(\tsn \varphi_k)^{-1}\right)
        \right|\\
        =&\left|\frac{1}{\nsn}\text{Tr}\left[\tsn \psi\cdot(\tsn \varphi)^{-1}\cdot \tsn(\varphi-\varphi_k)\cdot(\tsn \varphi_k)^{-1}\right]\right|\\
        \leq&\frac{\|\psi\|_\infty}{\delta^2}\cdot\frac{1}{\nsn}\text{Tr}\tsn(\varphi-\varphi_k)
        =\frac{\|\psi\|_\infty}{\delta^2}\cdot\|\varphi-\varphi_k\|_1.
    \end{aligned}\end{align}
    Moreover, it holds true
    \begin{align*}
            &\frac{1}{\nsn}\mathrm{Tr}\left|\tsn \psi\cdot(\tsn \varphi_k)^{-1}-\tsn\psi\cdot \tsn\varphi_k^{-1}\right|\\
            =&\frac{1}{\nsn}\mathrm{Tr}\left|\tsn \psi\cdot(\tsn \varphi_k)^{-1}\left(I_{\nsn}-\tsn\varphi_k\cdot \tsn\varphi_k^{-1}\right)\right|\\
            \leq&\frac{\|\psi\|_\infty}{\delta}\cdot\frac{1}{\nsn}\mathrm{Tr}\left|I_{\nsn}-\tsn\varphi_k\cdot \tsn\varphi_k^{-1}\right|\\
            \leq& \frac{\|\psi\|_\infty}{\delta}\cdot\left(\frac{1}{\nsn}\mathrm{Tr}\left|I_{\nsn}-\tsn\varphi_k\cdot \tsn\varphi_k^{-1}\right|^2\right)^{\frac{1}{2}}\\
            =&\frac{\|\psi\|_\infty}{\delta}\cdot\left(\frac{1}{\nsn}\mathrm{Tr}\left(I_{\nsn}- \tsn\varphi_k^{-1}\cdot\tsn\varphi_k\right)\cdot\left(I_{\nsn}-\tsn\varphi_k\cdot \tsn\varphi_k^{-1}\right)\right)^{\frac{1}{2}}
    \end{align*}
    By Lemma \ref{prod_lmm}, we have
    \begin{align*}
        \lim_{N\rightarrow+\infty}\frac{1}{\nsn}\mathrm{Tr}\left(I_{\nsn}- \tsn\varphi_k^{-1}\cdot\tsn\varphi_k\right)\cdot\left(I_{\nsn}-\tsn\varphi_k\cdot \tsn\varphi_k^{-1}\right)=0.
    \end{align*}
    Combined with Lemma \ref{prod_lmm}, it holds true
    \begin{equation}\label{lem_n_seconed}
        \lim_{N\rightarrow+\infty}\frac{1}{\nsn}\mathrm{Tr}\tsn \psi\cdot(\tsn \varphi_k)^{-1}=\int_{\mathbb{T}^d}\frac{\psi}{\varphi_k}~dm_d.
    \end{equation}
    From \eqref{lem_n_first} and \eqref{lem_n_seconed}, it follows that
    $$
    \limsup_{N\rightarrow+\infty}\left|\frac{1}{\nsn}\text{Tr} \left(\tsn \psi\cdot(\tsn \varphi)^{-1}\right)-\int_{\mathbb{T}^d}\frac{\psi}{\varphi_k}~dm_d
        \right|\leq \frac{\|\psi\|_\infty}{\delta^2}\cdot\|\varphi-\varphi_k\|_1.
    $$
     Letting $k\rightarrow+\infty$ and applying Lebesgue's dominated convergence theorem, we have
    $$
    \lim_{N\rightarrow+\infty}\frac{1}{\nsn}\text{Tr}\left(\tsn \psi\cdot(\tsn \varphi)^{-1}\right)=\int_{\mathbb{T}^d}\frac{\psi}{\varphi}~dm_d.
    $$
    When $n>1$, by Lemma \ref{prod_lmm}, it holds that
    $$
    \lim_{N\rightarrow+\infty}\frac{1}{\nsn}\text{Tr}\left[\tsn \psi\cdot\tsn \varphi^{-1}\right]^n=\int_{\mathbb{T}^d}\bigls\frac{\psi}{\varphi}\bigrs^n ~dm_d.
    $$
    We only need to prove that
    \begin{equation}\label{main_lmm_equa}
        \lim_{N\rightarrow+\infty}\frac{1}{\nsn}\text{Tr}\left(  \left[\tsn \psi\cdot(\tsn \varphi)^{-1}\right]^n-\left[\tsn \psi\cdot\tsn \varphi^{-1}\right]^n\right)=0.
    \end{equation}
   \noindent Set $A=\tsn \psi\cdot(\tsn \varphi)^{-1}$ and $B=\tsn \psi\cdot\tsn \varphi^{-1}$. Then
    $$
    \|A\|\leq \|\psi\|_\infty\delta^{-1},~~~\|B\|\leq \|\psi\|_\infty\delta^{-1}.
    $$
    Hence, as done in \cite{2021SZEGO,nikolski2021szego},
    \begin{align*}
        &\left|\frac{1}{\nsn}\text{Tr}\left(  \left[\tsn \psi\cdot(\tsn \varphi)^{-1}\right]^n-\left[\tsn \psi\cdot\tsn \varphi^{-1}\right]^n\right)\right|\\
        =&\left|\frac{1}{\nsn}\text{Tr}\left(\sum_{i=0}^{n-1}A^i(A-B)B^{n-1-i}\right)\right|\\
        \leq& C\frac{1}{\nsn}\text{Tr}\left|A-B\right|\\
        = &C\frac{1}{\nsn}\text{Tr}\left|\tsn\psi\cdot \left((\tsn \varphi)^{-1}-\tsn\varphi^{-1}\right)\right|\\
        \leq &C' \frac{1}{\nsn}\text{Tr}\left|\tsn\varphi^{-1}-(\tsn \varphi)^{-1}\right|\\
        \leq &C'\left(\frac{1}{\nsn}\text{Tr}\bigls\tsn\varphi^{-1}-(\tsn \varphi)^{-1}\bigrs^2\right)^{\frac{1}{2}}\\
        =&C'\left(\frac{1}{\nsn}\text{Tr}\bigls\tsn\varphi^{-1}\bigrs^2-\frac{2}{\nsn}\mathrm{Tr}\left(\tsn\varphi^{-1}\cdot(\tsn \varphi)^{-1}\right)+\frac{1}{\nsn}\mathrm{Tr}(\tsn \varphi)^{-2}\right)^{\frac{1}{2}}
    \end{align*}
    where $C$ and $C'$ are two constants that do not depend on $N$. By Lemma \ref{one_over_f}, Lemma \ref{prod_lmm}, and the case $n=1$, it follows that
\begin{align*}
    &\lim_{N\rightarrow+\infty}\left(\frac{1}{\nsn}\text{Tr}\bigls\tsn\varphi^{-1}\bigrs^2-\frac{2}{\nsn}\mathrm{Tr}\left(\tsn\varphi^{-1}\cdot(\tsn \varphi)^{-1}\right)+\frac{1}{\nsn}\mathrm{Tr}(\tsn \varphi)^{-2}\right)\\
    &=\int_{\mathbb{T}^d} \left(\frac{1}{\varphi^2}-\frac{2}{\varphi^2}+\frac{1}{\varphi^2}\right) dm_d=0.
\end{align*}
Therefore, we have \eqref{main_lmm_equa}, and the proof is completed.
\end{proof}

\begin{proof}[Proof of Theorem \ref{first_T_d}]
When $\varphi=0$, the conclusion follows easily. When $\varphi\neq 0$, without loss of generalization, we assume that $m_d(\{x:\varphi(x)>1\})>0$.
We first prove the theorem for $\varphi\in L^1_+(\mathbb{T}^d)$ with $\log\varphi\in L^1(\mathbb{T}^d)$.
From Theorem \ref{bdd_below_thm}, it holds that
\begin{equation}\label{l1_det_bound}
    0 \neq \exp\left(\int_{\mathbb{T}^d}\log\varphi~dm_d\right)\leq \bigls\det\tsn\varphi\bigrs^{\frac{1}{\nsn}}\leq\|\varphi\|_1.
\end{equation}
Suppose that $\alpha\geq\exp\left(\int_{\mathbb{T}^d}\log\varphi~dm_d\right)$ is a cluster point.
Hence, there exists a subsequence $\{\si_{N_k}\}$ such that
\begin{equation}\label{cluster_sub}
    \lim_{k\rightarrow +\infty}\bigls\det T_{\si_{N_k}}\varphi\bigrs^{\frac{1}{|\si_{N_k}|}}=\alpha,
\end{equation}
The remaining is to prove
$$
\alpha\leq\exp\left(\int_{\mathbb{T}^d}\log\varphi ~dm_d\right).
$$
 Let $\phi(x)=\max\{\varphi(x),1\}$ and $\psi=\phi-\varphi$. Then $\phi\in L^1_+(\mathbb{T}^d)$, $\phi \geq 1$ and hence $\log\phi\in L^1(\mathbb{T}^d)$.
 Moreover, $0\leq\psi\leq 1$.
Note that
$$\tsn\varphi=\tsn\phi-\tsn\psi=\left(\tsn\phi\right)^{\frac{1}{2}} \left(I_{\nsn}-\left(\tsn\phi\right)^{-\frac{1}{2}}\cdot\tsn\psi\cdot\left(\tsn\phi\right)^{-\frac{1}{2}}\right)\left(\tsn\phi\right)^{\frac{1}{2}}$$
and
 $$
\det\tsn\varphi=\det\tsn\phi\cdot\det\bigls I_{\nsn}-(\tsn\phi)^{-\frac{1}{2}}\cdot\tsn\psi\cdot(\tsn\phi)^{-\frac{1}{2}}\bigrs.
 $$
 Since $\phi\geq 1$ and $\phi\not\equiv 1$, we have $\tsn\phi>I_{\nsn}$, and it follows that
 \begin{equation}\label{norm_less_one}
     \|(\tsn\phi)^{-\frac{1}{2}}\cdot\tsn\psi\cdot(\tsn\phi)^{-\frac{1}{2}}\|<1.
 \end{equation}
 This means that $I_{\nsn}-(\tsn\phi)^{-\frac{1}{2}}\cdot\tsn\psi\cdot(\tsn\phi)^{-\frac{1}{2}}$ is strictly positive, and hence we have
\begin{equation}\label{main_thm_log_equa}
     \frac{1}{\nsn}\text{Tr}\log\tsn\varphi=\frac{1}{\nsn}\text{Tr}\log\tsn\phi+\frac{1}{\nsn}\text{Tr}\log\bigls I_{\nsn}-(\tsn\phi)^{-\frac{1}{2}}\cdot\tsn\psi\cdot(\tsn\phi)^{-\frac{1}{2}}\bigrs.
 \end{equation}
Since the operator norm and the trace-class norm are equivalent in matrix algebra with finite order, we have
\begin{align}\label{mani_thm_log_sum_equa}
    &\frac{1}{\nsn}\text{Tr}\log\tsn\phi-\frac{1}{\nsn}\text{Tr}\log\tsn\varphi \nonumber\\
    =&-\frac{1}{\nsn}\text{Tr}\log\bigls I_{\nsn}-(\tsn\phi)^{-\frac{1}{2}}\cdot\tsn\psi\cdot(\tsn\phi)^{-\frac{1}{2}}\bigrs \nonumber\\
    \stackrel{\text{by }\eqref{norm_less_one}}{=}&\sum_{n=1}^\infty\frac{1}{n\nsn}\text{Tr}\left[(\tsn\phi)^{-\frac{1}{2}}\cdot\tsn\psi\cdot(\tsn\phi)^{-\frac{1}{2}}\right]^n.
\end{align}
Let
$$
f_k(n)=\frac{1}{n\left|\si_{N_k}\right|}\text{Tr}\left[(T_{\si_{N_k}}\phi)^{-\frac{1}{2}}\cdot T_{\si_{N_k}}\psi\cdot(T_{\si_{N_k}}\phi)^{-\frac{1}{2}}\right]^n,~~~~k=1,2,\cdots,
$$
and
$$
f_0(n)=\int_{\mathbb{T}^d}\frac{\psi^n}{n\phi^n} ~dm_d.
$$
Note that for any $k\geq 1$,
$$
f_k(n)=\frac{1}{n\left|\si_{N_k}\right|}\text{Tr}\left[T_{\si_{N_k}}\psi\cdot(T_{\si_{N_k}}\phi)^{-1}\right]^n.
$$
Applying Corollary \ref{Folner_thm_first} and \eqref{cluster_sub}, we conclude from \eqref{mani_thm_log_sum_equa} that
\begin{align*}
    \|f_k\|_{l^1(\mathbb{N})}&=\sum_{n=1}^\infty \frac{1}{n\left|\si_{N_k}\right|}\text{Tr}\left[T_{\si_{N_k}}\psi\cdot(T_{\si_{N_k}}\phi)^{-1}\right]^n\\
    &=\frac{1}{\left|\si_{N_k}\right|}\text{Tr}\log T_{\si_{N_k}}\phi-\frac{1}{\left|\si_{N_k}\right|}\text{Tr}\log T_{\si_{N_k}}\varphi\\
    &\stackrel{k\to\infty}{\longrightarrow} \int_{\mathbb{T}^d} \log \phi~ dm_d-\log\alpha.
\end{align*}
By Lemma \ref{g_over_f}, $f_k$ convergences to $f_0$ pointwise as $k\rightarrow+\infty$. It gives that
\begin{align*}
    \int_{\mathbb{T}^d}-\log\left(1-\frac{\psi}{\phi}\right)dm_d&=\sum_{n=1}^{\infty}\int_{\mathbb{T}^d}\frac{\psi^n}{n\phi^n} ~dm_d=\sum_{n=1}^\infty f_0(n)\\
    &\leq\liminf_{k\rightarrow +\infty}\sum_{n=1}^{\infty}f_k(n)=\int_{\mathbb{T}^d} \log \phi~ dm_d-\log\alpha.
    \end{align*}
Then we have
$$
\alpha\leq \exp\left(\int_{\mathbb{T}^d}\log\varphi ~dm_d\right).
$$
This shows that Theorem \ref{first_T_d} holds  for $\varphi\in L^1_+(\mathbb{T}^d)$ with $\log\varphi\in L^1(\mathbb{T}^d)$.

If $\varphi\in L^1_+(\mathbb{T}^d)$ and $\log\varphi\notin L^1(\mathbb{T}^d)$, we consider $\varphi_c=\varphi+c$, $c>0$. Then $\log\varphi_c\in L^1(\mathbb{T}^d)$. Applying the above result gives
$$
\limsup_{N\rightarrow+\infty}\bigls\det\tsn\varphi\bigrs^{\frac{1}{\nsn}}\leq \limsup_{N\rightarrow+\infty}\bigls\det T_{\sn}\varphi_c\bigrs^{\frac{1}{\nsn}}= \exp\left(\int_{\mathbb{T}^d}\log(\varphi+c)~dm_d\right).
$$
Letting $c\rightarrow 0^+$, we have
$$
\lim_{N\rightarrow+\infty}\bigls\det\tsn\varphi\bigrs^{\frac{1}{\nsn}}=0,
$$
and the proof is completed.
\end{proof}

\section{Proofs of Theorem \ref{second_non_folner} and Theorem \ref{first_non_folner}}\label{sec_four}
We begin with the following lemma.

\begin{lem}\label{number_vol}
Let $m$ be a positive integer, $x_1,\dots ,x_m>0$, and $~t>0$. We define
$$
B(x_1,\dots,x_m;t):=\left\{(n_1,\dots,n_m)\in\mathbb{Z}^m_+:~\sum_{j=1}^m n_jx_j\leq t\right\}.
$$
Then there exists a constant $M=M(x_1,\dots,x_m)$ depending only on $x_1,\dots,x_m$, such that
$$
-MQ_m(t)\leq\left|B(x_1,\dots,x_m;t)\right|-\frac{t^m}{m!\prod_{j=1}^m x_j}\leq MQ_m(t),
$$
where $Q_m(t)=1+t+\dots+t^{m-1}$.
\end{lem}
\begin{proof}
We prove it by induction. When $m=1$, we have
$$
|B( x_1;t)|=\left[\frac{t}{x_1}\right]+1=\frac{t}{x_1}+O(1),
$$
where $[x]$ represents the largest integer not greater than $x$.
Assume that the lemma holds for $m=s-1$. Then there exists a constant $M$ depending only on $x_1,\dots,x_{s-1}$ such that
$$
-MQ_{s-1}(t)\leq |B(x_1,\dots,x_{s-1};t)|-\frac{t^{s-1}}{(s-1)!\prod_{j}^{s-1}x_j}\leq MQ_{s-1}(t),~~~~\forall~t>0.
$$
For $m=s$, since
$$
|B(x_1,\dots,x_s;t)|=\sum_{r=0}^{\left[\frac{t}{x_s}\right]}|B(x_1,\dots,x_{s-1};t-rx_s)|,
$$
and by induction, we have
\begin{equation}\label{first_estimate}
    \left||B(x_1,\dots,x_s;t)|-\frac{1}{(s-1)!\prod_{j}^{s-1}x_j}\sum_{r=0}^{\left[\frac{t}{x_s}\right]}\big(t-rx_s\big)^{s-1}\right|
    \leq M\sum_{j=0}^{s-2}\sum_{r=0}^{\left[\frac{t}{x_s}\right]}\big(t-rx_s\big)^j.
\end{equation}
It is easy to check that for any $j\geq 0$
$$
\sum_{r=0}^{\left[\frac{t}{x_s}\right]}(t-rx_s)^{j}=\frac{t^{j+1}}{(j+1)x_s}+O(Q_{j+1}(t)).
$$
Combined with \eqref{first_estimate}, there exist a constant $M'$ such that
\begin{equation*}
    \Big| |B(x_1,\dots,x_s;t)|-\frac{1}{s!\prod_{j=1}^s x_j}t^{s}\Big|\leq M' Q_s(t).
\end{equation*}
Hence the lemma follows by induction.
\end{proof}

Let $k$ be a fixed positive integer and define $Y_N$  to be
$$
Y_N=
\{(n_1,\dots,n_k)\in\mathbb{Z}^k_+:~p_1^{n_1}\cdots p_k^{n_k}\leq N\},
$$
then $\{Y_N\}$ is a F{\o}lner sequence of $\mathbb{Z}^k$ (see \cite{nikolski2021szego}).

\begin{lem}\label{dn_folner_zk} The following limit exists.
\begin{equation}\label{DN}
    \lim_{N\rightarrow +\infty} \frac{\left| Y_N\right|}{(\log N)^k}=\frac{1}{k!\prod_{j}^k \log p_j}.
\end{equation}
\end{lem}
\begin{proof}
    By definition, $Y_N=B(\log p_1,\dots,\log p_k; \log N)$.
    Thus \eqref{DN} is a direct conclusion from Lemma \ref{number_vol}.
\end{proof}

We denote by $E_k$ the set of positive integers that are coprime with the first $k$ primes, i.e.
$$
E_k=\{m\in \mathbb{N}: (m, p_1\cdots p_k)=1\}.
$$

The next lemma is essential for the proof of Proposition \ref{szego_thm_k_vari}.
\begin{lem}\label{con_lem}
Let $k$ be a fixed positive integer and suppose that the sequence $\{a_N\}_{N=0}^\infty$ is a real sequence with $a_0=0$.
If the sequence $\{\frac{a_N}{(\log N)^k}\}_{N=2}^\infty$ converges as $N\rightarrow +\infty$, then the sequence $\{\frac{1}{N}\sum_{m\in E_k}a_{[N/m]}\}_{N=0}^\infty$ converges. Furthermore,
$$
\lim_{N\rightarrow +\infty}\frac{1}{N}\sum_{m\in E_k}a_{\left[\frac{N}{m}\right]} = \frac{(p_1-1)\cdots (p_k-1)}{p_1 \cdots p_k} \int_{0}^{1} a_{\left[\frac{1}{x}\right]} ~dx.
$$
\end{lem}
\begin{proof}
   Let $g(x)=a_{[1/x]}$ for $0<x\leq 1$.
   Since $\{\frac{a_N}{(\log N)^k}\}$ converges, there exist $M, N_0>0$ such that
   $|a_N| \leq M(\log N)^k$ for $N>N_0$. Take $\delta=\frac{1}{N_0+1}$, and hence we have
   $$
   |g(x)|\leq M|\log x|^k~~~\text{for}~0<x<\delta.
   $$
   So the function $g(x)$ is  Riemann integrable.
   Let $q=p_1\cdots p_k$, and $F_q=\{m\in \mathbb{N}: m\leq q, (m, q)=1\}$. Then
   $| F_q|=\varphi(q)=(p_1-1)\cdots (p_k-1)$,
   where $\varphi$ is Euler's totient function.
   Observe that
   \begin{align*}
       \frac{1}{N}\sum_{m\in E_k}a_{\left[\frac{N}{m}\right]}
       =&\frac{1}{N}\sum_{m\in E_k}g\left(\frac{m}{N}\right)
       =\frac{1}{N}\sum_{r=0}^\infty\sum_{v\in F_q}g\left(\frac{rq+v}{N}\right)\\
       =&\frac{1}{q}\sum_{v\in F_q}\left(\frac{q}{N}\sum_{r=0}^{\big[\frac{N-v}{q}\big]}g\left(\frac{rq+v}{N}\right)\right).
   \end{align*}
   For each $v\in F_q$, $\frac{q}{N}\sum_{r=0}^{\big[\frac{N-v}{q}\big]}g\left(\frac{rq+v}{N}\right)$ is a Riemann sum and converges to $\int_0^1 g(x)~dx$ as $N\rightarrow+\infty$. Hence
   \begin{align*}
       \lim_{N\rightarrow+\infty}\frac{1}{N}\sum_{m\in E_k}a_{\left[\frac{N}{m}\right]}
   =\frac{|F_q|}{q}\int_0^1g(x)~dx
   =\frac{(p_1-1)\cdots(p_k-1)}{p_1\cdots p_k}\int_0^1 a_{\left[\frac{1}{x}\right]}~dx.
   \end{align*}

\end{proof}

To prove Theorem \ref{second_non_folner}, we first consider the case when $\varphi$ depends only on finitely many variables.
\begin{thm}
\label{szego_thm_k_vari}
    Assume that $\varphi\in \lifr(\mathbb{T^{\infty}})$ and $\varphi$ depends only on the first $k$ variables, that is, $\varphi(\bm{z})=\phi(z_1,\dots,z_k)$, where $\phi\in\lifr(\mathbb{T}^k)$. Then there exists a probability measure $\mu$ concentrating on a countable subset of $[\mathrm{ess\text{-}inf}~\varphi,\mathrm{ess\text{-}sup}~\varphi]$ such that
for any $f\in C[\mathrm{ess\text{-}inf}~\varphi,\mathrm{ess\text{-}sup}~\varphi]$,
$$
\lim_{N\rightarrow +\infty} \frac{1}{N}\mathrm{Tr} f \big(\bm{T_N \varphi}\big)=\int f~d\mu.
$$
\end{thm}
\begin{proof}
Recall that $\si_N=\{1,\dots,N\}$.
    Let $E_k$ be defined as above and for any $m\in E_k$, we define
    $$
    F_{m,N}=\left\{mp_1^{s_1}\cdots p_k^{s_k}:~s_j\geq 0\right\}\cap\sn.
    $$
    Hence $\sn=\bigsqcup_{m\in E_k}F_{m,N}$. Since $\varphi$ depends only on the first $k$ variables, this shows that $\bm{\widehat{\varphi}}\big(j/i\big)=\widehat{\varphi}(\alpha(j/i))\neq0$ only happen in the case $\alpha\big(j/i\big)\in \mathbb{Z}^k\subset\mathbb{Z}^\infty$.
    It follows that if $i,j \in \sn$ lie in different $F_{m,N}$,
    $\bm{\widehat{\varphi}}\big(j/i\big)=0$. Thus,
    $$
    \bm{T_N \varphi} \cong\bigoplus_{m\in E_k}\bm{T_{F_{m,N}}\varphi}.
    $$
    Assume that $F_{1,[N/m]}=\{x_1,\dots,x_t\}$, and hence $F_{m,N}=\{mx_1,\dots,mx_t\}$. From    $\bm{\widehat{\varphi}}\left(mx_j/mx_i\right)=\bm{\widehat{\varphi}}\left(x_j/x_i\right)$,
    we have
    $$
    \bm{T_{F_{m,N}}\varphi}\cong \left\{\bm{\widehat{\varphi}}\left(\frac{mx_j}{mx_i}\right)\right\}_{i,j=1}^{t}=\left\{\bm{\widehat{\varphi}}\left(\frac{x_j}{x_i}\right)\right\}_{i,j=1}^{t}\cong \bm{T_{F_{1,[N/m]}}{\varphi}},
    $$
    and it follows that
    $$
    \bm{T_N \varphi}\cong\bigoplus_{m\in E_k} \bm{T_{F_{1,[N/m]}}{\varphi}}.
    $$
    Let $Y_N$ be defined as in Lemma \ref{dn_folner_zk}. Then for any $t\in\mathbb{N}$,
    $$
    \bm{T}_{F_{1,t}}\bm{\varphi}\cong T_{Y_t}\phi.
    $$
    This shows
    $$
    \bm{T_N\varphi}\cong\bigoplus_{m\in E_k} T_{Y_{[N/m]}}\phi,
    $$
    and hence for any $f\in C[\text{ess-inf~}\varphi,\text{ess-sup~}\varphi]$, we have
    $$
    \frac{1}{N}\text{Tr} f\left(\bm{T_N\varphi} \right)=\sum_{m\in E_k}\frac{1}{N} \text{Tr} f\left(T_{Y_{[N/m]}}\phi \right).
    $$
    Write $a_N=\text{Tr} f\left(T_{Y_N}\phi\right)$ and $a_0=0$. The limit
    $\lim_{N\to\infty}\frac{1}{| Y_{N}|} \text{Tr} f\left(T_{Y_{N}}\phi\right)$ exists since $\{Y_N\}_{N=1}^\infty$ is a F{\o}lner sequence of $\mathbb{Z}^k$.
     Moreover,
    because the sequence $\{| Y_{N}|/(\log N)^k\}_{N=2}^{\infty}$ convergences by Lemma \ref{dn_folner_zk}, the limit
    $\lim_{N\rightarrow\infty}a_N/(\log N)^k$ exists and hence from Lemma \ref{con_lem}
    $$
    \lim_{N\rightarrow\infty}\frac{1}{N}\text{Tr} f\left(\bm{T_N\varphi} \right)~~~~\text{exists}.
    $$
    It is easy to see that the functional $f\mapsto\lim_{N\rightarrow\infty}\frac{1}{N}\text{Tr} f\left(\bm{T_N\varphi}\right)$ is bounded on $C[\text{ess-inf~}\varphi,\text{ess-sup~}\varphi]$. By Riesz representation theorem, there is a regular Borel measure $\mu$ on $[\text{ess-inf~}\varphi,\text{ess-sup~}\varphi]$ such that for each $f\in C[\text{ess-inf~}\varphi,\text{ess-sup~}\varphi]$,
    $$\lim_{N\rightarrow\infty}\frac{1}{N}\text{Tr} f\left(\bm{T_N\varphi}\right)=\int f~d\mu.$$
    Obviously, $\mu$ is a probability measure. From Lemma \ref{con_lem}, we have
    \begin{equation}\label{unsatisfied}
    \int f~d\mu
    =\frac{(p_1-1)\cdots (p_k-1)}{p_1\cdots p_k}\sum_{n=1}^{\infty}\frac{\mathrm{Tr} f\left(T_{Y_n}\phi\right)}{n(n+1)}.
    \end{equation}
    Let $\{\lambda_{j,n}\}_{j=1}^{|Y_n|}$ be the eigenvalues of $T_{Y_n}\phi$ counted with multiplicities. Then
    $$
    \mu=\frac{(p_1-1)\cdots (p_k-1)}{p_1\cdots p_k}\sum_{n=1}^{\infty}\frac{\left|Y_n\right| }{n(n+1)}\left( \frac{1}{|Y_n|}\sum_{j=1}^{|Y_n|}\delta_{\lambda_{j,n}}\right),
    $$
    and hence $\mu$ concentrates on a countable subset.
\end{proof}
\begin{rem}\label{lim_kvariable}
Taking $f=1$ in \eqref{unsatisfied}, one obtains a number theoretic equality
$$
\sum_{n=1}^{\infty}\frac{\left|Y_n\right|}{n(n+1)}=(1-\frac{1}{p_1})^{-1}\cdots (1-\frac{1}{p_k})^{-1}.
$$
In particular, when $k=1$, it holds true
$$\sum_{n=1}^{\infty}\frac{[\log_2 n]}{n(n+1)}=1.$$
\end{rem}

\begin{lem}\label{norm_two_controll}
Let $\varphi\in\lif(\ti)$. Then
$$
\frac{1}{N}\left\|\bm{T_N\varphi}\right\|^2_{\mathcal{S}_2}\leq\|\varphi\|_2^2,
$$
where $\|~\|_{\mathcal{S}_2}$ denotes the Hilbert-Schmidt norm.
\end{lem}
\begin{proof}
    By the definition of the Hilbert-Schmidt norm,
    \begin{align*}
        \frac{1}{N}\left\|\bm{T_N \varphi}\right\|^2_{\mathcal{S}_2}
        =&\frac{1}{N}\sum_{1\leq i,j\leq N}\left|\bm{\widehat{\varphi}}\left(\frac{j}{i}\right)\right|^2\\
        \leq& \frac{1}{N} \sum_{i=1}^{N} \left(\sum_{q\in\mathbb{Q}_+}\left|\bm{\widehat{\varphi}}\left(q\right)\right|^2 \right) =\sum_{q\in\mathbb{Q}_+}\left|\bm{\widehat{\varphi}}\left(q\right)\right|^2
        =\|\varphi\|^2_2.
    \end{align*}
\end{proof}

For any $k\geq 0$, $\mathbb{Z}^k$ can be viewed as a subgroup of $\mathbb{Z}^\infty$. Let
$\mathcal{F}_k$ denote the smallest $\sigma$ algebra of Borel subsets of $\ti$ for which the first $k$ coordinate functions are measurable. Given $\varphi\in\lifr(\ti)$, we define
$$
\varphi_k=E(\varphi|\mathcal{F}_k)=\sum_{\kappa\in \mathbb{Z}^k}\widehat{\varphi}(\kappa)\bm{z}^{\kappa},~~~~k=1,2,\dots.
$$
By some techniques of martingale theory \cite{durrett2019probability}, it follows that
\begin{itemize}
    \item[(i)]
    $\varphi_k\in\lifr(\ti)$ and $\|\varphi_k\|_\infty\leq \|\varphi\|_\infty;$
    \item[(ii)] $\lim_{k\rightarrow+\infty}\|\varphi_k-\varphi\|_2^2 \rightarrow 0$.
\end{itemize}

\begin{proof}[Proof of Theorem \ref{second_non_folner}]
It is enough to prove that for any $f\in C[\text{ess-inf~}\varphi,\text{ess-sup~}\varphi]$,
\begin{equation}\label{con_f}
    \lim_{N\rightarrow +\infty} \frac{1}{N}\mathrm{Tr} f \big(\bm{T_N \varphi}\big)~~~~\text{exists}.
\end{equation}
   By Lemma \ref{general_prop}, we only need to prove \eqref{con_f} holds for $f(x)=x^m,~ \forall~ m\geq 0$.
   When $m=0$, $f$ is a constant and the conclusion is obvious. When $m\geq 1$, we have for any $k=1, 2, \ldots,$
   $$
    \left(\bm{T_N\varphi}\right)^m-\left(\bm{T_N\varphi_k}\right)^m=\sum_{i=0}^{m-1}\left(\bm{T_N\varphi_k}\right)^i \left(\bm{T_N\varphi}-\bm{T_N\varphi_k}\right)\left(\bm{T_N\varphi}\right)^{m-i-1}.
    $$
   Hence as done in the proof of Lemma \ref{g_over_f}, and by Lemma \ref{norm_two_controll},
        \begin{align*}\label{main_ineq}
        &\left|\frac{1}{N}\text{Tr}\left(\bm{T_N\varphi}\right)^m-\frac{1}{N}\text{Tr}\left(\bm{T_N\varphi_k}\right)^m\right|\\
        = &\frac{1}{N}\left|\text{Tr}\left(\sum_{i=0}^{m-1}\left(\bm{T_N\varphi_k}\right)^i \left(\bm{T_N\varphi}-\bm{T_N\varphi_k}\right)\left(\bm{T_N\varphi}\right)^{m-i-1}\right)\right|\\
        \leq &\frac{1}{N}\sum_{i=0}^{m-1}\|\bm{T_N\varphi_k}\|^i\|\bm{T_N\varphi}\|^{m-i-1} \text{Tr}\left|\bm{T_N(\varphi-\varphi_k)}\right|\\
        \leq &\frac{m\|\varphi\|_\infty^{m-1}}{N}\text{Tr}\left|\bm{T}_{N}(\bm{\varphi-\varphi_k})\right|\\
        \leq& m\|\varphi\|_\infty^{m-1}\frac{\|\bm{T}_{N}(\bm{\varphi-\varphi_k})\|_{\mathcal{S}_2}}{\sqrt{N}}\\
        \leq& m\|\varphi\|_\infty^{m-1}\|\varphi-\varphi_k\|_2.
    \end{align*}
    For any $\varepsilon>0$, take $k_0$ large enough such that
    $$
    \|\varphi-\varphi_{k_0}\|_2\leq \frac{\varepsilon}{m\|\varphi\|_\infty^{m-1}}.
    $$
    Then the above argument gives that
    \begin{equation}\label{nk0inequ}
        \left|\frac{1}{N}\text{Tr}\left(\bm{T_N\varphi}\right)^m-\frac{1}{N}\text{Tr}\left(\bm{T_N\varphi_{k_0}}\right)^m\right|<\varepsilon,~~~\forall N>0.
    \end{equation}
    Since by Theorem \ref{szego_thm_k_vari}, $\left\{ \frac{1}{N} \text{Tr} \big(\bm{T_N\varphi_{k_0}}\big)^m\right\}_{N=1}^{\infty}$ is a Cauchy sequence, there exists $N_0$ such that for any $N,M >N_0$,
    \begin{equation}\label{k0inequ}
     \left|\frac{1}{N}\text{Tr}\left(\bm{T_N\varphi_{k_0}}\right)^m-\frac{1}{M}\text{Tr}\left(\bm{T_M\varphi_{k_0}}\right)^m\right|<\varepsilon.
    \end{equation}
     Combined \eqref{nk0inequ} with \eqref{k0inequ}, a standard technique
     shows that
    $$
\left|\frac{1}{N}\text{Tr}\left(\bm{T_N\varphi}\right)^m-\frac{1}{M}\text{Tr}\left(\bm{T_M\varphi}\right)^m\right|<3\varepsilon.
$$
Then $\lim\limits_{N\rightarrow+\infty} \frac{1}{N} \text{Tr}\left(\bm{T_N\varphi}\right)^m$ exists and the proof is completed.
\end{proof}

Now, Theorem \ref{first_non_folner} follows easily.

\begin{proof}[Proof of Theorem \ref{first_non_folner}]
    By Theorem \ref{second_non_folner}, $\left\{\sn=\{1,\dots,N\}\right\}$ is a hypo-F{\o}lner sequence. The conclusion follows from Theorem \ref{convergence_L_one}.
\end{proof}

\section{Gram determinants of dilated systems}\label{sec_five}
We apply our results to prove a theorem about Gram determinants of dilation systems.

Let $H$ be a Hilbert space with an orthonormal basis $\{e_n\}_{n\in\mathbb{N}}$. For each $m\in \mathbb{N}$, we define an isometric operator $T_m$ by $T_me_n=e_{mn}$ for $n=1,2,\cdots.$ Then  $\mathcal{T}=\{T_n\}_{n\in\mathbb{N}}$ satisfies  $T_{mn}=T_mT_n$ for all $m$ and $n$. Given a nonzero vector $h\in H$, denote by $G(h,\mathcal{T})$ the Gram matrix $\{\langle T_i h,T_j h\rangle\}_{i,j\in\mathbb{N}}$, then   $G(h,\mathcal{T})$ is a multiplicative Toeplitz matrix. For any finite subset $\si\subset\mathbb{N}$, write $G_\si(h,\mathcal{T})=\{\langle T_i h,T_j h\rangle\}_{i,j\in\si}$.  We  calculate the determinant of $G_\si(h,\mathcal{T})$ via the following unitary transform from $H$ onto $H^2(\ti)$. This  unitary transform is defined by  $U:H\rightarrow H^2(\ti)$, $e_n\mapsto\bm{z}^{\alpha(n)}$. It is easy to see  $UT_n=M(\bm{z}^{\alpha(n)})U$ for any $n\in\mathbb{N}$. Therefore we have
$$
\langle T_i h,T_j h\rangle=\langle UT_i h,UT_j h\rangle=\langle M(\bm{z}^{\alpha(i)})Uh,M(\bm{z}^{\alpha(j)})Uh\rangle=\widehat{\bm{|Uh|^2}}\left(\frac{j}{i}\right),
$$
that is,  $G(h,\mathcal{T})=\bm{T(|Uh|^2)}$. Applying Theorem \ref{bdd_below_thm}, Theorem \ref{first_multi} and Theorem \ref{first_non_folner}, we have the following results.
\begin{prop}\label{dilation}
\begin{enumerate}[(1)]
    \item For any finite subset $\si\subset\mathbb{N}$,
    $$  0\neq\exp\left(\int_{\ti}\log\left|Uh\right|^2~dm_{\infty}\right)\leq\left(\det G_\si(h,\mathcal{T})\right)^{\frac{1}{|\si|}}\leq \|h\|^2.
    $$
    \item For any multiplicative F{\o}lner sequence $\{\si_N\}$,
    $$
    \lim_{N\rightarrow\infty}\left(\det G_{\si_N}(h,\mathcal{T})\right)^{\frac{1}{|\si_N|}}=\exp\left(\int_{\ti}\log\left|Uh\right|^2~dm_{\infty}\right)\neq 0.
    $$
\end{enumerate}
\end{prop}
\begin{proof}
    It  suffices to prove that
    $$
    \exp\left(\int_{\ti}\log\left|Uh\right|^2~dm_{\infty}\right)\neq 0.
    $$
    Since $Uh\in H^2(\ti)$, then $\log|Uh|^2\in L^1(\ti)$ by \cite[Corollary 2]{aleman2019fatou}, and the proof is completed.
\end{proof}

Let us see two specific examples. The first example is related to the periodic dilation problem, which attracts a lot of attention due to its closed connection with the cyclicity problem in $H^2(\ti)$ (see \cite{hedenmalm1995hilbert,noor2019, nikolski2012shadow, dan}). Setting   $H=L^2(0,1)$, we  consider the  canonical  orthonormal basis of $L^2(0,1)$,  $\{\sqrt{2}\sin(n\pi x):n\in\mathbb{N}\}$, and define $D_m(\sqrt{2}\sin(n\pi x))=\sqrt{2}\sin(mn\pi x)$ for any $m$ and $n$.
For $\varphi\in L^2(0,1)$, consider $\varphi$ as defined on the whole real axis by extending it to an odd $2$-periodic function. It is easy to see that $D_m\varphi(x)=\varphi(mx)$.
Applying Proposition \ref{dilation} to the dilation system $\mathcal{D}=\{D_n:n\in\mathbb{N}\}$, we  generalize \cite[Corollary 3.5]{2021SZEGO} to  the following situation.
\begin{cor}
Let $\varphi\in L^2(0,1)$. Then
    \begin{enumerate}[(1)]
    \item for any finite subset $\si\subset\mathbb{N}$,
    $$  0\neq\exp\left(\int_{\ti}\log\left|U\varphi\right|^2~dm_{\infty}\right)\leq\left(\det G_\si(\varphi,\mathcal{D})\right)^{\frac{1}{|\si|}}\leq \|\varphi\|^2.
    $$
    \item for any multiplicative F{\o}lner sequence $\{\si_N\}$,
    $$
    \lim_{N\rightarrow\infty}\left(\det G_{\si_N}(\varphi,\mathcal{D})\right)^{\frac{1}{|\si_N|}}=\exp\left(\int_{\ti}\log\left|U\varphi\right|^2~dm_{\infty}\right)\neq 0.
    $$
\end{enumerate}
\end{cor}

The second example comes from power dilations (see \cite{nikolski2012shadow,dan2023,nikolski2017}). Setting  $H^2_0(\mathbb{T})=\{f\in H^2(\mathbb{T}):\widehat{f}(0)=0\}$, we consider  the  canonical  orthonormal basis of $H^2_0(\mathbb{T})$, $\{w, w^2,\cdots\}$, and define  $C_mw^n=w^{mn} $ for any $m$ and $n$, that is, $C_mf(w)=f(w^m)$ for any $f\in H^2_0(\mathbb{T}).$
Using Proposition \ref{dilation} to the power dilation system $\mathcal{C}=\{C_n:n\in\mathbb{N}\}$, we have the following results.
\begin{cor}
Let $f\in H^2_0(\mathbb{T})$. Then
    \begin{enumerate}[(1)]
    \item for any finite subset $\si\subset\mathbb{N}$,
    $$  0\neq\exp\left(\int_{\ti}\log\left|Uf\right|^2~dm_{\infty}\right)\leq\left(\det G_\si(f,\mathcal{C})\right)^{\frac{1}{|\si|}}\leq \|f\|^2.
    $$
    \item for any multiplicative F{\o}lner sequence $\{\si_N\}$,
    $$
    \lim_{N\rightarrow\infty}\left(\det G_{\si_N}(f,\mathcal{C})\right)^{\frac{1}{|\si_N|}}=\exp\left(\int_{\ti}\log\left|Uf\right|^2~dm_{\infty}\right)\neq 0.
    $$
\end{enumerate}
\end{cor}

\section*{Acknowledgements}
This work was partially supported by NNSF of China (12231005) and NSF
of Shanghai (21ZR1404200). 

\bibliographystyle{amsalpha}
\bibliography{name}

\end{document}